\documentclass[12pt,a4paper]{amsart}

\usepackage{amsmath}
\usepackage{amssymb}

\theoremstyle{plain}   
\begingroup 

\newtheorem{theorem}{Theorem}[section]   
\newtheorem{corollary}[theorem]{Corollary}     
\newtheorem{lemma}[theorem]{Lemma}         
\newtheorem{proposition}[theorem]{Proposition}  
\endgroup

\theoremstyle{definition}
\newtheorem{definition}[theorem]{Definition}   

\theoremstyle{remark}
\newtheorem{remark}[theorem]{Remark}        

\numberwithin{equation}{section}

\newcommand{\ep}{\varepsilon}

\newcommand{\R}{{\mathbb R}}
\newcommand{\N}{{\mathbb N}}

\newcommand{\vf}{\varphi}

\newcommand{\mcA}{{\mathcal{A}}}

\usepackage[dvips]{color}

\begin{document}

\title{Hadamard differentiability via G\^ ateaux differentiability}

\thanks{The research was supported by the grant GA\v CR P201/12/0436.}

\author{Lud\v ek Zaj\'\i\v{c}ek}



\email{zajicek@karlin.mff.cuni.cz}

\address{Charles University,
Faculty of Mathematics and Physics,
Sokolovsk\'a 83,
186 75 Praha 8-Karl\'\i n,
Czech Republic}


\begin{abstract} 

Let $X$ be a separable Banach space, $Y$ a Banach space and $f: X \to Y$ a mapping. 
 We prove that there exists a  $\sigma$-directionally porous set $A\subset X$ such that if $x\in X \setminus A$,  $f$ is Lipschitz at $x$, and $f$ is  G\^ateaux  differentiable at $x$, then $f$ is Hadamard differentiable at $x$.
 If $f$ is Borel  measurable (or has the Baire property) and  is G\^ ateaux differentiable at all points, then $f$ is Hadamard differentiable at all points
  except a  set which is $\sigma$-directionally porous set (and so is Aronszajn null, Haar null and $\Gamma$-null).
   Consequently, an everywhere G\^ ateaux differentiable $f: \R^n \to Y$ is Fr\' echet differentiable except a
    nowhere dense $\sigma$-porous set.
    \end{abstract}
   
\markboth{L.~Zaj\'{\i}\v{c}ek}{Hadamard differentiability via G\^ateaux differentiability}

\maketitle

\section{Introduction}
 Hadamard derivative of a mapping $f:X \to Y$ between Banach spaces, which is stronger than G\^ ateaux derivative but weaker that Fr\' echet derivative, was applied many times in the literature. For three formally different but equivalent definitions of Hadamard derivative see
  Lemma \ref{eh} below. If $X$ is finite-dimensional, then Hadamard derivative coincides with Fr\' echet derivative, but
   for infinite-dimensional spaces Fr\' echet derivative is ``much stronger'', even for Lipschitz $f$.
   
  On the other hand, if $f$ is Lipschitz, then Hadamard derivative coincides with G\^ ateaux derivative; in this sense
   these two types of derivatives are rather close. Using this fact,  we obtain the well-known result that if $f$ is everywhere G\^ ateaux differentiable and has the Baire property, than $f$ is  Hadamard differentiable at all points except a nowhere dense set.
            Indeed, \cite[Corollary 3.9]{Shk} easily implies that, under above assumptions, 
      $f$ is locally Lipschitz on a dense open set $U$.  However, an easy example (see Remark \ref{neli}) shows   that in general (even if $X=Y=\R$)   we cannot find $U$ with 
    ``measure   null'' complement.

      We prove  (Theorem \ref{vsha}) that if $X$ is separable and $f$ has the Baire property and is everywhere G\^ ateaux differentiable, then $f$ is Hadamard differentiable at all points except a $\sigma$-directionally porous set. This is an interesting additional information, since each  $\sigma$-directionally porous subset of a
 separable Banach spaces $X$ is ``measure null'': it is Aronszajn (=Gauss) null, and so also Haar null, (see \cite[p. 164 and  Chap. 6]{BL}) and
  also $\Gamma$-null (see \cite[Remark 5.2.4] {LPT}). 
  
  As an easy consequence of Theorem \ref{vsha},  we obtain  (see Theorem \ref{kondim}) that an everywhere G\^ ateaux differentiable $f: \R^n \to Y$ is Fr\' echet differentiable except a
    $\sigma$-porous set.
  
  The main ingredient of the proof of Theorem \ref{vsha} is Theorem \ref{lgaha}. It asserts that if $X$ is separable and 
 $f:X \to Y$, then  there exists a  $\sigma$-directionally porous set $A\subset X$ 
  such that if $x\in X \setminus A$,  $f$ is Lipschitz at $x$, and $f$ is  G\^ateaux  differentiable at $x$, then $f$ is Hadamard differentiable at $x$.

  Lemma  \ref{dkuh}  implies that if $X$ is separable and $f$  has the Baire property and  is everywhere
 G\^ ateaux differentiable, then $f$ is Lipschitz at all points except a $\sigma$-directionally porous set. So,
  Theorem \ref{lgaha} together with  Lemma  \ref{dkuh} imply Theorem \ref{vsha}.
  
  Further, Theorem \ref{lgaha} shows that the infinite-dimensional version of Stepanoff theorem on G\^ ateaux differentiability from \cite{Du} holds also for Hadamard differentiability (see Corollary \ref{step}).

 \section{Preliminaries} 
 
 In the following, by a Banach space we mean a real Banach space. If $X$ is a Banach space, we set
    $S_X:= \{x \in X: \|x\|=1\}$. The symbol $B(x,r)$ will denote the open ball with center $x$ and radius $r$.

In a metric space $(X,\rho)$, the system of all sets with the Baire property is the smallest $\sigma$-algebra containing all open sets and all first category sets. We will say that a {\it mapping $f: (X, \rho_1) \to (Y, \rho_2)$ has the Baire property} if $f$ is measurable with respect to the $\sigma$-algebra of all sets with the Baire property.
 In other words, $f$ has the Baire property, if and only if $f^{-1}(B)$ has the Baire property for all Borel sets $B \subset Y$ (see \cite[Section 32]{Ku}).

    Let $X$, $Y$ be Banach spaces, $G \subset X$ an open set, and $f:G \to Y$ a mapping. 
    
   We say  that {\it $f$ is Lipschitz at $x \in G$} if  $\limsup_{y \to x} \frac{\|f(y)-f(x)\|}{\|y-x\|} < \infty$.
   We say  that $f$ is {\it pointwise Lipschitz} if $f$ is Lipschitz at all points of $G$.

   The directional and one-sided  directional derivatives 
    of $f$ at $x\in G$ in the direction $v\in X$ are defined respectively by
   $$f'(x,v) := \lim_{t \to 0} \frac{f(x+tv)-f(x)}{t}\ \ \text{and}\ \  f'_+(x,v) := \lim_{t \to 0+} \frac{f(x+tv)-f(x)}{t}.$$
    The {\it Hadamard directional and one-sided  directional derivatives} 
    of $f$ at $x\in G$ in the direction $v\in X$ are defined respectively by
   $$f'_H(x,v) := \lim_{z \to v, t \to 0} \frac{f(x+tv)-f(x)}{t}\ \ \text{and}\ \  f'_{H+}(x,v) := \lim_{z \to v, t \to 0+} \frac{f(x+tv)-f(x)}{t}.$$

   It is easy to see that $f'(x,v)$ (resp. $f'_H(x,v)$) exists if and only if  $f'_+(x,-v) = -f'_+(x,v)$
    (resp. $f'_{H+}(x,-v) = - f'_{H+}(x,v)$). 
    
    It is well-known and easy to prove that, if $f$ is locally Lipschitz on $G$, then  $f'(x,v)=f'_H(x,v)$ (resp. $f'_+(x,v)=f'_{H+}(x,v)$) whenever one of these two derivatives exists.
 
 The usual modern definition of the Hadamard derivative is the following:
 
 A continuous linear operator $L: X \to Y$ is said to be a {\it Hadamard derivative} of $f$ at a point $x \in X$ if
  $$ \lim_{t \to 0} \frac{f(x+tv)-f(x)}{t} = L(v)\ \ \  \text{for each}\ \ \ v \in X$$
   and the limit is uniform with respect to $v \in C$, whenever $C\subset X$ is a compact set.
    In this case we set $f'_H(x) := L$. 
   
   The following fact is well-known (see \cite{Sha}):
   
   \begin{lemma}\label{eh}
   Let $X$, $Y$ be Banach spaces, $\emptyset \neq G \subset X$  an open set, $x \in G$, $f: G \to Y$ a mapping and
    $L: X \to Y$ a continuous linear operator. Then the following conditions are equivalent:
    \begin{enumerate}
    \item
    $f'_H(x) := L$,
    \item
    $f'_H(x,v)= L(v)$ for each $v \in X$,
    \item
    if $\vf: [0,1] \to X$ is such that $\vf(0)=x$ and $\vf_+'(0)$ exists, then $(f \circ \vf)_+'(0)= L(\vf_+'(0))$.
    \end{enumerate}
    \end{lemma}
    
    Recall (see, e.g. \cite{Sha}), that if  $f$ is Hadamard differentiable at $a \in G$, then $f$ is G\^ ateaux
     differentiable at $a$, and 
    \begin{equation}\label{lipgha}
    \text{if $f$ is locally Lipschitz on $G$, then also the opposite implication holds.}
     \end{equation}
     Further, 
    \begin{equation}\label{konhf}
    \text{if $X= \R^n$, then Hadamard differentiability is equivalent to Fr\' echet differentiability.}
     \end{equation}
    
    \begin{definition}
    Let $X$ be a Banach space. We say that $A \subset X$ is {\it directionally porous at a point $x \in X$}, if there exist
     $0 \neq v \in X$, $p>0$  and a sequence  $t_n \to 0$ of positive real numbers such that $B(x+  t_n v, p t_n) \cap A = \emptyset$. (In this case we say that $A$ is  {\it porous at $x$ in  direction $v$}.)
      
      We say that $A \subset X$ is {\it directionally porous} if $A$ is directionally porous at each point $x \in A$. 
      
      We say that $A \subset X$ is {\it $\sigma$-directionally porous} if it is a countable union of directionally porous
       sets.
       \end{definition}
       
       Recall that directional porosity is stronger than (upper) porosity, but 
      \begin{equation}\label{pdp}       
       \text{ if $X$ is finite-dimensional, then these two notions are equivalent.}
       \end{equation}

       Let $X$ be a Banach space, $x \in X$, $v \in S_X$ and $\delta >0$. Then we define the open cone
     $C(x,v,\delta)$ as the set of all $y \neq x$ for which  $\|v - \frac{y-x}{\|y-x\|}\| < \delta$.

The following easy inequality is well known (see e.g.\ \cite[Lemma~5.1]{MS}):
\begin{equation}\label{triangle}
\text{if }u,v\in X\setminus\{0\},\text{ then }
 \left\|\frac{u}{\|u\|}-\frac{v}{\|v\|}\right\|
\leq \frac{2}{\|u\|}\, \|u-v\|.
\end{equation}

Because of the lack of a reference we supply the proof of the following easy fact.
\begin{lemma}\label{hnli}
 Let $X$ be a Banach space, $Y$ a Banach space, $G \subset X$ an open set, $a \in G$, and $f:G \to Y$ a mapping.
  Then the following are equivalent.
  \begin{enumerate}
  \item
    $f'_{H+}(a,0)$ exists,
    \item
    $f'_{H}(a,0)$ exists,
    \item
    $f'_{H}(a,0)=0$,
    \item
    $f$ is Lipschitz at $a$.
    \end{enumerate}
\end{lemma}
\begin{proof}
The implications $(ii) \Rightarrow (i)$ and $(iii) \Rightarrow (ii)$ are trivial.

To prove $(iv) \Rightarrow (iii)$, suppose that $K>0$, $\delta>0$ and $\|f(x)-f(a)\| \leq K \|x-a\|$ for each $x \in B(a,\delta)$. To prove $(iii)$, let $\ep \in (0,K)$ be given. Then, for $z \in X$ with $\|z\|< \ep/K$ and $0< |t| < \delta$, we obtain
$$\|f(a+tz)-f(a)\| \leq    K |t| \frac{\ep}{K}, \ \ \ \ \  \left\|\frac{f(a+tz)-f(a)}{t}\right\|  \leq \ep,$$
 and $(iii)$ follows.
 
 To prove $(i) \Rightarrow (iv)$, suppose that $(iv)$ does not hold. Then, for each $n \in \N$, choose $h_n \in X$
  such that  $0 < \|h_n\| < n^{-2}$ and $\|f(a+h_n) - f(a)\| \geq n^2 \|h_n\|.$ Set  $t_n:= n \|h_n\|$ and $z_n:= n^{-1}\|h_n\|^{-1} h_n$.
   Then $z_n \to 0$, $t_n \to 0$, $t_n >0$, and
   $$ \left\|\frac{f(a+t_nz_n) - f(a)}{t_n}\right\| = \frac{\|f(a+h_n) - f(a)\|}{n\|h_n\|} \geq n,$$
    which clearly implies that  $f'_{H+}(a,0)$ does not exist.
\end{proof}

We will need also the following special case of \cite[Lemma 2.4]{Shk}. It can be proved by the Kuratowski-Ulam theorem
 (as is noted in \cite{Shk}), but the proof given in \cite{Shk} is more direct.
 
 \begin{lemma}\label{skarin}
 Let $U$ be an open subset of a Banach space $X$. Let $M \subset U$ be a set residual in $U$ and $z \in U$. Then there exists a line $L \subset X$ such that $z$ is a point of accumulation of $M \cap L$.
 \end{lemma}

   \section{Results}   
    
    \begin{proposition}\label{cll}
    Let $X$ be a separable Banach space, $Y$ a Banach space, $G \subset X$ an open set, and $f:G \to Y$ a mapping. Let
     $A$ be the set of all points $x \in G$ for which there exists a cone $C= C(x,v,\delta)$ such that
     $\limsup_{y \to x, y \in C} \frac{\|f(y)-f(x)\|}{\|y-x\|} < \infty$ and $f$ is not Lipschitz at $x$.
      Then $A$  is a  $\sigma$-directionally porous  set.
       \end{proposition}
    \begin{proof}
    Let $\{v_n: n \in \N\}$ be a dense subset of $S_X$. For natural numbers  $k$, $p$, $n$, $m$ denote
     by $A_{k,p,n,m}$ the set of all points $x \in A$ such that
     \begin{equation}\label{trij}
     \frac{\|f(y)-f(x)\|}{\|y-x\|} < k\ \ \text{whenever}\ \ 0< \|y-x\|< \frac{1}{p}\ \ \text{and}\ \ 
      \left\| \frac{y-x}{\|y-x\|}- v_n\right\| < \frac{1}{m}.
     \end{equation}
     Since clearly $A$ is the union of all sets $A_{k,p,n,m}$, it is sufficient to prove that each set
    $A_{k,p,n,m}$ is directionally porous. So, suppose that natural numbers  $k$, $p$, $n$, $m$ and 
    $x \in A_{k,p,n,m}$ are given. It is sufficient to prove that $A_{k,p,n,m}$ is directionally porous
     at $x$. 
     
     To this end, find a sequence $y_i \to x$ such that 
     \begin{equation}\label{star}
     \text{ $\frac{\|f(y_i)-f(x)\|}{\|y_i-x\|} > k(12 m +4)$ for each
      $i \in \N$.}
      \end{equation}
       Set  $r_i:= \|y_i-x\|$ and $x_i:= x- 6mr_iv_n$. It is sufficient to prove that there exists
     $i_0 \in \N$ such that
     \begin{equation}\label{prmn}
      B(x_i,r_i) \cap A_{k,p,n,m} = \emptyset\ \ \ \text{for each}\ \ \ i \geq i_0.
     \end{equation}
     To this end,  consider $i \in \N$ and $z_i \in  B(x_i,r_i) \cap A_{k,p,n,m}$.
     
     Observe that
     $$ 6mr_i -r_i \leq \|x-x_i\|- \|x_i-z_i\| \leq \|x-z_i\| \leq \|x-x_i\|+ \|x_i-z_i\| \leq 6mr_i +r_i, $$ 
     $$6mr_i -2r_i \leq\|x-z_i\| -r_i \leq \|y_i-z_i\|   \leq\|x-z_i\| + r_i  \leq 6mr_i +2r_i.$$
     These inequalities imply that there exists $i_0 \in \N$ (independent on $i$)  such that
     \begin{equation}\label{jlm}
     0<  \|x-z_i\| < \frac{1}{p}\ \ \text{and}\ \ 0<  \|y_i-z_i\| < \frac{1}{p},\ \ \ \ \text{if}\ \  i \geq i_0.
     \end{equation}
     Applying  \eqref{triangle}, we obtain
     \begin{equation}\label{nabla}
     \left\| \frac{x-z_i}{\|x-z_i\|}- v_n\right\| = \left\| \frac{x-z_i}{\|x-z_i\|}-\frac{x-x_i}{\|x-x_i\|}\right\|
      \leq 2\frac{\|x_i-z_i\|}{6mr_i -r_i}< \frac{2 r_i}{6mr_i -r_i} \leq \frac{1}{m}\ \ \text{and}
     \end{equation}
     \begin{equation}\label{ctyr}
     \left\| \frac{y_i-z_i}{\|y_i-z_i\|}- v_n\right\| = \left\| \frac{y_i-z_i}{\|y_i-z_i\|}-\frac{x-x_i}{\|x-x_i\|}\right\| \leq 2 \frac{\|y_i-x\| +\|x_i-z_i\|}{\|y_i-z_i\|}
      < 2\frac{2r_i}{6mr_i -2r_i} \leq \frac{1}{m}.
     \end{equation}
     Since $z_i \in  A_{k,p,n,m}$, conditions \eqref{trij}, \eqref{jlm}, \eqref{nabla} and \eqref{ctyr} imply that, if $i \geq i_0$, then
     $$ \frac{\|f(x)-f(z_i)\|}{\|x-z_i\|} < k\ \ \ \text{and}\ \ \  \frac{\|f(y_i)-f(z_i)\|}{\|y_i-z_i\|} < k.$$
     So,  using also \eqref{star}, we obtain that, if $i \geq i_0$, then 
     \begin{multline*}
     r_i k (12m+4) \leq  \|f(x)-f(y_i)\| \leq      \|f(x)-f(z_i)\| + \|f(y_i)-f(z_i)\| \\
     < k \|x-z_i\| + k \|y_i-z_i\| \leq k (6mr_i +r_i +6mr_i +2r_i),
         \end{multline*}
     which is impossible. So we have proved \eqref{prmn}.
         \end{proof}

 \begin{proposition}\label{smerdh}
 Let $X$ be a separable Banach space, $Y$ a Banach space, $G \subset X$ an open set, and $f: G \to Y$ a mapping. Let $M$ be the set of all
  $x \in G$ at which $f$ is Lipschitz  and there exists $v\in X$ such that     $f'_+(x,v)$ exists but $f'_{H+}(x,v)$
   does not exist. Then $M$ is $\sigma$-directionally porous.
   \end{proposition}
   \begin{proof}
   For each $k \in \N$, set
   $$M_k:= \{x \in M:\ \|f(y)-f(x)\| \leq k \|y-x\|\ \ \text{ whenever}\ \ \|y-x\| <1/k\}.$$
   It is clearly sufficient to prove that each $M_k$ is directionally porous. So suppose that $k \in \N$
    and $x \in M_k$. Choose $v \in X$ such that     $f'_+(x,v)=:y$ exists but $f'_{H+}(x,v)$
   does not exist. Lemma \ref{hnli} gives that $v \neq 0$.  We will prove that $M_k$ is porous at $x$ in direction $v$.
   Since $f'_{H+}(x,v)$
   does not exist, we can choose $\ep>0$ such that for each $\delta>0$ there exists $w \in B(v,\delta)$ and $t \in (0,\delta)$
    with $ \|t^{-1}(f(x+tw)-f(x)) - y\| > \ep$.
    
    For each $n \in \N$, we choose $w_n \in B(v,1/n)$ and $t_n \in (0,1/n)$ such that
    \begin{equation}\label{vnep}
     \|(t_n)^{-1}(f(x+t_n w_n)-f(x)) - y\| > \ep.  
     \end{equation}
        Set  $x_n:= x +t_n v$ and $p: = \ep (3k)^{-1}$. It is sufficient to prove that there exists $n_0 \in \N$
     such that
     \begin{equation}\label{praz}
     B(x_n, pt_n) \cap M_k  = \emptyset\ \ \text{for each}\ \ n \geq n_0.
         \end{equation}
    To this end consider $n \in \N$ and $z_n \in  B(x_n, pt_n) \cap M_k$. Denote $x_n^*:= x + t_n w_n$. 
    We have
    \begin{equation}\label{hvd}
     \|z_n -x_n\| < pt_n,\ \ \ \|z_n-x_n^*\| \leq \|x_n- x_n^*\|+ \|z_n - x_n\| \leq t_n/n + p t_n.
     \end{equation}
    Thus there exists $\widetilde{n_0} \in \N$ (independent on $n$) such that the inequality $n \geq \widetilde{n_0}$
     implies  $\|z_n -x_n\| < 1/k$, $\|z_n-x_n^*\|< 1/k$, and thus, since $z_n \in M_k$, also
     \begin{equation}\label{trj}
     \|f(x_n)- f(z_n)\| \leq k \|x_n - z_n\|\ \ \text{and}\ \ \|f(x_n^*) - f(z_n)\| \leq k \|x_n^* - z_n\|.
     \end{equation}
    If $n \geq \widetilde{n_0}$, then  \eqref{vnep}, \eqref{hvd} and \eqref{trj} imply
    \begin{eqnarray*}
    \ep&<&\left\|\frac{f(x_n^*)-f(x)}{t_n}- y\right\|    \\ 
    &\leq& \left\|\frac{f(x_n)-f(x)}{t_n}- y\right\| + \frac{\|f(x_n)-f(z_n)\|}{t_n} +  \frac{\|f(x_n^*)-f(z_n)\|}{t_n}\\
    &\leq& \left\|\frac{f(x_n)-f(x)}{t_n}- y\right\| + kp + k (1/n+p)\\
     &=& \left\|\frac{f(x_n)-f(x)}{t_n}- y\right\| + (2/3)\ep + k/n.
    \end{eqnarray*}
    Since $\left\|\frac{f(x_n)-f(x)}{t_n}- y\right\| \to 0$, we obtain a contradiction, if $n>n_0$, where
     $n_0 \geq  \widetilde{n_0}$ is sufficiently large (independent on $n$). So \eqref{praz} is proved.
     \end{proof}
     
     \begin{remark}\label{twos}
     The corresponding ``two-sided'' result which works with $f'(x,v)$ and $f'_H(x,v)$ clearly follows from
     Proposition   \ref{smerdh}. 
          \end{remark}
     
     Using Proposition \ref{smerdh} (together with Remark \ref{twos}) and Lemma \ref{eh}, we immediately obtain:
     
     \begin{theorem}\label{lgaha}
     Let $X$ be a separable Banach space, $Y$ a Banach space, $G \subset X$ an open set and $f: G \to Y$ a mapping.
     Then the set of  all points at
   which $f$ is Lipschitz and G\^ ateaux differentiable but is not Hadamard differentiable is  $\sigma$-directionally   porous. 
   \end{theorem}
   
   Theorem \ref{lgaha} together with \cite[Corollary 3.9]{Shk} have the following consequence.
   
   \begin{corollary}\label{plgh}
   Let $X$ be a separable Banach space, $Y$ a Banach space, $G \subset X$ an open set, and $f:G \to Y$ a 
    pointwise Lipschitz mapping. Then the set of  all points from $G$ at
   which $f$ is G\^ ateaux differentiable but is not Hadamard differentiable is nowhere dense and  $\sigma$-directionally   porous. 
   \end{corollary}
     
     \begin{lemma}\label{dkuh}
     Let $X$ be a  separable Banach space, $Y$ a Banach space, $G \subset X$ an open set, and $f:G \to Y$ a mapping.
     Suppose that $f$ has the Baire property and, for each line $L \subset X$ with $G \cap L \neq \emptyset$,
     the restriction of $f$ to $G \cap L$ is continuous. Denote by $A$ the set of all points $x \in G$ for which
      there exists a set $B_x \subset S_X$ of the second category in $S_X$ such that
      \begin{equation}\label{omdk}
       \limsup_{t \to 0+} \frac{\|f(x+tv)-f(x)\|}{t} < \infty\ \ \ \text{for each}\ \ \ v \in B_x
       \end{equation}
      and $f$ is not Lipschitz at $x$. Then $A$ is a $\sigma$-directionally porous set.
          \end{lemma}
    \begin{proof}
    By Proposition \ref{cll} it is sufficient to prove that for each $x \in A$ there exists $v \in S_X$
     and $\delta>0$ such that
     \begin{equation}\label{pruh}
     \limsup_{y \to x, y \in C(x,v,\delta)} \frac{\|f(y)-f(x)\|}{\|y-x\|} < \infty.
          \end{equation}
       So fix a point $x \in A$. Find $p_0 \in \N$ with $B(x, 1/p_0) \subset G$ and, for natural numbers  $p \geq p_0$,   $k$,  denote
       $$ S_{p,k}:= \{v \in S_X: \frac{\|f(x+tv)-f(x)\|}{t} \leq k\ \ \ \text{for each}\ \ \ 0<t < 1/p.\}$$ 
       Since $B_x$ is clearly covered by all sets  $S_{p,k}$, we can find $k$ and  $p \geq p_0$ such that
       $ S_{p,k}$ is a second category set (in $S_X$).
       
       Without any loss of generality, we can clearly suppose that $x=0$ and $f(x)=f(0)=0$.
              We will show that $S:= S_{p,k}$ has the Baire property  (in $S_X$).  Set  $E:= B(0,1/p) \setminus \{0\}$ and $E^*:= S_X \times (0,1/p)$ (equipped with the product topology). For each $(v,t) \in E^*$ set $\vf((v,t)):= tv$. Then $\vf: E^* \to E$ is clearly a homeomorphism
       (with $\vf^{-1}(z) = (z/\|z\|, \|z\|)$ for $z \in E$). Set
       $$M:= \{z \in E: \frac{\|f(z)\|}{\|z\|} \leq k\} = \{z\in E: \|f(z)\| - k\|z\| \leq 0\}.$$
       Since $f$ has the Baire property, using continuity of norms we easily obtain that the real function
        $z \mapsto  \|f(z)\| - k\|z\|$ has the Baire property on $E$, and so $M$ has the Baire property in $E$.
        Consequently $M^*:= \vf^{-1}(M)$ and $C:= E^* \setminus M^*$ have the Baire property in $E^*$.
        
        For any set $A \subset E^*$ and $v \in S_X$ define the section $A_v:= \{t \in (0,1/p): (v,t) \in A\}$ and the projection $\pi(A):= \{v \in S_X:\ A_v \neq \emptyset\}$. It is easy to see that
        \begin{equation}\label{dopl}
         S = S_X \setminus \pi(C).
         \end{equation}
         Using the continuity of $f$ on the set $\{tv: t \in (0,1/p)\}$, we obtain that
         \begin{equation}\label{ope}
         C_v \ \ \text{is open for each}\ \  v \in S_X.
         \end{equation}
    Since  $C$ has the Baire property, we can write $C= H \cup T$, where $H$ is a $G_{\delta}$ set (in $E^*$)
     and $T$ is a first category set (in $E^*$). We have  
     \begin{equation}\label{sjr}
     \pi(C) = \pi(H) \cup \pi(T) = \pi(H) \cup(\pi(T)\setminus \pi(H)).
     \end{equation}
     The set $\pi(H)$ is analytic (see, e.g \cite[Exercise 14.3]{Ke}), and so (see \cite[Theorem 21.6]{Ke}) it has the Baire property. By Kuratowski-Ulam theorem
     (see, e.g. \cite[Theorem 8.41]{Ke}) there exists a first category set $Z \subset S_X$ such that $T_v$ is a first category subset of $(0,1/p)$
     for each $v \in S_X \setminus Z$.
     
      The inclusion $\pi(T)\setminus \pi(H) \subset Z$ holds. Indeed, suppose
      on the contrary that there exists $v \in  (\pi(T)\setminus \pi(H)) \setminus Z$. Then $T_v$ is a nonempty first category set and $H_v = \emptyset$. Thus $C_v = H_v \cup T_v = T_v$. Using \eqref{ope}, we obtain that $C_v$ is a nonempty open first category set, which is a contradiction. 
      
      Consequently $\pi(T)\setminus \pi(H)$ is a first category set and thus $S$ has the Baire property
       by \eqref{dopl} and \eqref{sjr}. Since $S$ is also a second category set, we can choose $v \in S_X$ and
        $\delta>0$ such that $S=S_{p,k}$ is residual in $S_X \cap B(v,\delta)$.
         Set $U^*:= (S_X \cap B(v,\delta)) \times (0,1/p)$, $\psi:=\vf\restriction_{U^*}$ and $U:= C(0,v,\delta) \cap B(0,1/p)$. Then $\psi:U^*\to U$ is clearly a homeomorphism and $\psi(S \times (0,1/p))\subset M$. Since $S \times (0,1/p)$ is residual in $U^*$ (see \cite[Lemma 8.43]{Ke}), we obtain that $M$ is residual in $U$. Now consider an arbitrary $z \in U$.
         By Lemma \ref{skarin} there exists a line $L \subset X$ and points $z_n \in M \cap L\cap U$ with $z_n \to z$. 
         Since the restriction of $f$ to $L \cap U$ is continuous, we obtain $\frac{\|f(z_n)\|}{\|z_n\|}\to 
       \frac{\|f(z)\|}{\|z\|}$, and consequently $z \in M$. So $U \subset M$, which implies \eqref{pruh}.
          \end{proof}
   In fact, we will use only the following immediate consequence of Lemma \ref{dkuh}.       
  \begin{lemma}\label{dkuhs}
     Let $X$ be a separable Banach space, $Y$ a Banach space, $G \subset X$ an open set, and $f:G \to Y$ a mapping.
     Suppose that $f$ has the Baire property and, for each line $L \subset X$ with $G \cap L \neq \emptyset$,
     the restriction of $f$ to $G \cap L$ is continuous. Denote by $B$ the set of all points $x \in G$ at which
     $f$ is G\^ ateaux differentiable 
      and $f$ is not Lipschitz at $x$. Then $B$ is a $\sigma$-directionally porous set.
          \end{lemma}

As an immediate application of Lemma \ref{dkuhs} and Theorem \ref{lgaha} we obtain

\begin{proposition}\label{prgh}
 Let $X$ be a separable Banach space, $Y$ a Banach space, $G \subset X$ an open set, and $f:G \to Y$ a mapping.
     Suppose that $f$ has the Baire property and, for each line $L \subset X$ with $G \cap L \neq \emptyset$,
     the restriction of $f$ to $G \cap L$ is continuous. Then the set of all points $x\in G$ at which $f$ is 
 G\^ ateaux differentiable but is not Hadamard differentiable is  $\sigma$-directionally   porous. 
 \end{proposition}
 
 As already mentioned in Introduction, \cite[Corollary 3.9]{Shk} easily implies that if $f$ is everywhere G\^ ateaux differentiable and  has the Baire property, then
      $f$ is locally Lipschitz on a dense open set $U$, and so (see \eqref{lipgha}) $f$ is Hadamard differentiable at all points except a nowhere dense set (even if $X$ is nonseparable).
 
 This result together with Proposition \ref{prgh} give the following theorem.
          
   \begin{theorem}\label{vsha}
 Let $X$ be a separable Banach space, $Y$ a Banach space, $G \subset X$ an open set. Suppose that $f:G \to Y$ has the Baire property and is
   everywhere G\^ ateaux differentiable. Then $f$ is Hadamard differentiable at all points  of $G$ except a nowhere dense $\sigma$-directionally porous set.
   \end{theorem}
   
    \begin{remark}\label{neli}
     The assumptions of Theorem \ref{vsha} do not imply that $f$ is locally Lipschitz on the complement of a closed 
     $\sigma$-directionally porous set even in the case $X=Y=\R$.
     
     Indeed, let $(a_n,b_n),\ n\in \N$ be the system of bounded intervals such that $G:= \bigcup _{n \in \N} (a_n,b_n)$
      is dense in $\R$ and   $F:= \R \setminus  G$  has positive Lebesgue measure. We can clearly define a function $f$ on $\R$ such that
     \begin{enumerate}
     \item $f(x)=0$ for $x \in F$,
     \item  $f$ has continuous derivative on each $(a_n,b_n)$,
     \item $|f(x)| \leq (x-a_n)^2 (x-b_n)^2$ for $x \in (a_n,b_n)$,
     \item  for each $n\in \N$ and $\delta>0$, the derivative $f'$ is unbouded both on $(a_n, a_n + \delta)$ and on
      $(b_n-\delta, b_n)$.
      \end{enumerate}
     Then $f$ is everywhere differentiable, $G$ is the largest open set on which $f$ is locally Lipschitz and $F$ is not $\sigma$-directionally porous.  So Theorem \ref{vsha} cannot be proved simply using only the fact that G\^ ateaux and Hadamard differentiability coincide for locally Lipschitz functions.
      \end{remark}
   
   \begin{remark}\label{gadi}
   In Theorem \ref{vsha}, we cannot omit the assumption that $f$ has the Baire property. This follows from Lemma \ref{hnli} and Shkarin's \cite{Shk} example which shows that on each separable infinite-dimensional Banach space
    there exists an everywhere G\^ ateaux differentiable real function which is discontinuous at all points.
     \end{remark}
    
    However, if $X$ is finite-dimensional, the assumption that $f$ has the Baire property can be omitted. Indeed (as is also  noted without a reference in
     \cite[Remark 3.5]{Shk}), if $f$ is everywhere G\^ ateaux differentiable on $\R^n$, than $f$ has the Baire property. This follows, e.g., from the old 
       well-known fact (see \cite{Le}) that a partially continuous real function on $\R^n$ is in the $(n-1)^{st}$
      Baire class (whose proof works also for a Banach space valued function). So, using also \eqref{konhf} and \eqref{pdp}, we obtain the following result.
      
   \begin{theorem}\label{kondim}
   Let $G$ be an open subset of $\R^n$, $Y$ a Banach space,  and $f:G \to Y$ an
   everywhere G\^ ateaux differentiable  mapping. Then $f$ is Fr\' echet differentiable at all points  of $G$ except a nowhere dense $\sigma$-porous set.
   \end{theorem}

   Using  Theorem \ref{lgaha} and the main result of \cite{Du} (which generalizes the result of \cite{PZ} and
    improves the result of \cite{Bo}), we easily obtain Theorem \ref{zldu} below. In its formulation, the system
     $\tilde{\mcA}$ of subsets of a separable Banach space is used. This system which was defined in \cite{PZ} is strictly smaller (see \cite[Proposition 13]{PZ}) than the well-known system of Aronszajn null sets (and also than the system of $\Gamma$-null sets, see \cite{Za}). We will not recall the (slightly technical) definition of the system  $\tilde{\mcA}$. Note only, that all members of  $\tilde{\mcA}$ are Borel by definition and that it is easy to see that $\tilde{\mcA}$ is stable with respect to countable unions.
     
\begin{theorem}\label{zldu}     
 Let $X$ be a separable Banach space, $Y$ a Banach space with the Radon-Nikod\' ym property, $G \subset X$ an open set, and $f:G \to Y$ a mapping.
     Then there exists a set $A \in \tilde{\mcA}$ such that if $x\in G \setminus A$ and $f$ is Lipschitz at $x$, then
      $f$ is Hadamard differentiable at $x$.
  \end{theorem}
  \begin{proof}
  The main result of \cite{Du} says that  there exists a set $A_1 \in \tilde{\mcA}$ such that if $x\in G \setminus A_1$ and $f$ is Lipschitz at $x$, then
      $f$ is G\^ ateaux differentiable at $x$. Let $A_2^*$ be the set of all $x\in G$ at which $f$ is Lipschitz,
       G\^ ateaux differentiable, but not Hadamard differentiable. Then  $A_2^*$ is $\sigma$-directionally porous
       by  Theorem \ref{lgaha}. Using \cite[Proposition 14]{PZ} and \cite[Theorem 12]{PZ} we easily obtain that
        there exists $A_2 \in \tilde{\mcA}$ which contains $A_2^*$. Now it is  clearly sufficient to set $A: =A_1 \cup
         A_2$.
         \end{proof}
         
         As a corollary, we obtain an infinite dimensional analogue of Stepanoff theorem on Hadamard differentiability.
         
    \begin{corollary}\label{step}     
 Let $X$ be a separable Banach space, $Y$ a Banach space with the Radon-Nikod\' ym property, $G \subset X$ an open set, and $f:G \to Y$ a  pointwise Lipschitz mapping.
     Then 
      $f$ is Hadamard differentiable at all points of $G$ except a set belonging to $\tilde{\mcA}$.
  \end{corollary}

   \begin{remark}\label{fdav}      
         
      Corollary \ref{step} and \eqref{konhf} imply the following result.
     \smallskip
   
   {\it Let  $Y$ a Banach space with the Radon-Nikod\' ym property, $G \subset \R^n$ an open set, and $f:G \to Y$ a pointwise Lipschitz mapping. Then 
      $f$ is  Fr\' echet differentiable at all points of $G$ except a set belonging to $\tilde{\mcA}$.}
      \smallskip
      
     However, it is not probably an improvement of the classical Stepanoff theorem, since
      from some unpublished results follows that each Borel Lebesgue null subset of $\R^n$ belongs to $\tilde{\mcA}$. 
       For $n=2$ this follows from    \cite[Theorem 7.5]{ACP} (via \cite[Theorem 12]{PZ}), which is stated without a proof, and  for $n\geq 3$ 
        from the corresponding theorem recently announced by  M. Cs\"ornyei and P. Jones.
   \end{remark}

 \end{document}